\documentclass[12pt,a4,twoside]{amsart}
\usepackage{amssymb}
\usepackage{latexsym}

\usepackage{epsfig}
\usepackage{srcltx}
\usepackage{hyperref}
\usepackage{pdfpages}
\usepackage{amsmath}
\usepackage[hang]{subfigure}

\theoremstyle{plain}
\newtheorem{theorem}{Theorem}

\newtheorem{proposition}{Proposition}

\theoremstyle{definition}
\newtheorem{definition}{Definition}
\newtheorem{example}{Example}

\theoremstyle{plain}
\newtoks\thehProclaim
\newtheorem*{Proclaim}{\the\thehProclaim}

\theoremstyle{definition}
\newtoks{\thehRemark}
\newtheorem*{Remark}{\the\thehRemark}

\renewcommand{\leq}{\leqslant}
\renewcommand{\geq}{\geqslant}

\begin{document}

\title{Extremal Area of Polygons, sliding along a circle}

\author{Dirk Siersma}

\address{Utrecht University, Department of Mathematics, P.O.Box 80.010, 3508TA, Utrecht, The Netherlands}
\email{d.siersma@uu.nl}

\subjclass[2000]{58K05, 52B99}

\keywords{Area, polygon, ellipse, critical point, Morse index}
\date{\today}

\begin{abstract}
We determine all critical configurations for the Area function on polygons with vertices on a circle or an ellipse. For isolated critical points we compute their Morse index, resp index of the gradient vector field. We relate the computation at an isolated degenerate point to an eigenvalue question about combinations. In the even dimensional case non-isolated singularities occur as `zigzag trains' .

\end{abstract}

\maketitle

\date{\today}

\section{Introduction}
A traditional problem for the maximal area of a polygon is the Isoperimetric Problem. It dates back to the antiquity. Many mathematicians where involved in this problem, e.g. Steiner; see e.g. the historical overview of   Bl\r{a}sj\" o  \cite{isoper}.

\smallskip

Several problems of maximal area with other constraints  have been  studied; many of them concern computer science, e.g. as maximal triangle in a convex polygon.

\smallskip

Recently there became more attention for the other critical points of Area  as a function on a configuration space. A goal  is a Morse-theoretic approach: determine all critical points, check if they are Morse and if so compute the Morse index. In specific problems critical points and indices can be successfully  described by nice geometric properties.

This has been carried out  e.g.  for the  signed Area function on  linkages with given edge length  \cite{KP},  \cite{PZ} and more recently in the context of the isoperimetric problem \cite{KPS}.

\smallskip

We continue this approach in a very elementary case:  A polygon with $n$ vertices on a circle. We allow self-intersections and coinciding vertices. We use the signed Area function $sA$. This has the advantage and disadvantage that due  to orientation we can have negative area and in case of multiple covered regions we count with multiplicities.

Our main theorem \ref{t:mainthm} gives geometric criteria for the critical points and determines also the Hesse matrix at those points. Most of the critical points are of Morse type and look as a regular star, but several  of them have zigzag behaviour. The Morse index is determined  by  combinatorial data.  The remaining cases are degenerate: All points coincide (degenerate star) which has an isolated critical point in case $n$ is odd; and zigzag-trains, which are non-isolated in case $n$ is even.

We compute also the index of the gradient vector field at the degenerate star by Euler-characteristic arguments. In the last section we discuss the Eisenbud-Levine-Khimshiashvili method to calculate this index. This relates nicely to a combinatorial question.

Note, that the problem of extremal area polygons in an ellipse is also solved due to the existence of an area preserving affine map.

This paper is a by-product of a study (in progress) of polygons with vertices sliding along a set of curves. This was triggered by the work of Wilson \cite{wilson} in 1917.

Part of the work was done at Luminy. The author thanks CIRM for the excellent working conditions, Wilberd van der Kallen for his advice around Mathematica computations, George Khimshiashvili and Gaiane Panina for useful discussions.

\section{Polygons in a circle}
In this section we give a complete description of critical polygons and compute indices at isolated zero's of the gradient vector field.

\smallskip
We first state the formula for signed Area $sA$.
Take unit circle with center $O$. We fix the point $P_1$ of the polygon with vertices $P_1, \cdots , P_n$ and work with the reduced configuration space  $(S^1)^{n-1}$.
We use $\alpha_i = \angle P_iOP_{i+1}$ as coordinates ($i=1, \cdots, n-1$).
Coordinates are modulo $2 \pi$, we choose values in the interval $(-\pi,\pi]$. $\alpha_n =\angle P_nOP_1 =2 \omega \pi - \sum_{i=1}^{n-1}  \alpha_i$.

The signed area function is $$\emph{sA} = \sum_{i=1}^{n} \sin \alpha_i,$$

\medskip

We next describe some special types of polygons (see also Figure \ref{fig:stars}):
\begin{definition} 
Let all $|\alpha_i |$  be  equal.\\ 
A {\em regular star} is a polygon with vertices on a circle and all angles $\alpha_i$ are equal and different from $0$ and $\pi$.
In case all $\alpha_i = \pi$ we call it a {\em complete fold}.\\
Regular stars are determined by the winding number $\omega$ of the polygon with respect to the center of the circle.  All angles are equal to   $\frac{2 \pi \omega}{n}$. \\
A {\em zigzag} is a polygon with vertices on a circle and all $|\alpha_i |$ are equal and different from zero, but not all $\alpha_i$ have the same sign.
Let $f$ be the number of forwards edges ($0 <\alpha_i  < \pi$) and $b$ the number of backward edges ($\alpha_i < 0$)  in the  polygon.  NB. $f+b=n$.
We define a number $m = [\frac{n-1}{2}]$; so $n=2m+1$ if $n$ is odd and $n=2m+2$ if $n$ is even.\\
In case  $f \ne b$  and $|\alpha_i|= \frac{2 \omega  \pi}{|f-b|}$, $\omega \ne 0$  
we call them {\em zigzag stars}. Zigzag stars are determined by a sequence of signs of $\alpha_i$.  Their realization in the plane is a $|f-b|$-regular star with some multiple edges.\\
In case $f=b$  gives rise to a 1-dimensional singularities, since $|\alpha_i|$ can be arbitrary;
we call them {\em zigzag trains}.  They only occur if the number of vertices is even. In that case they include the degenerate star and complete fold.

A {\em degenerate star} is the polygon with all $\alpha_i$ equal to zero.

\end{definition}

\begin{figure}[ht]
	\centering
	\includegraphics[width=1.05 \textwidth]{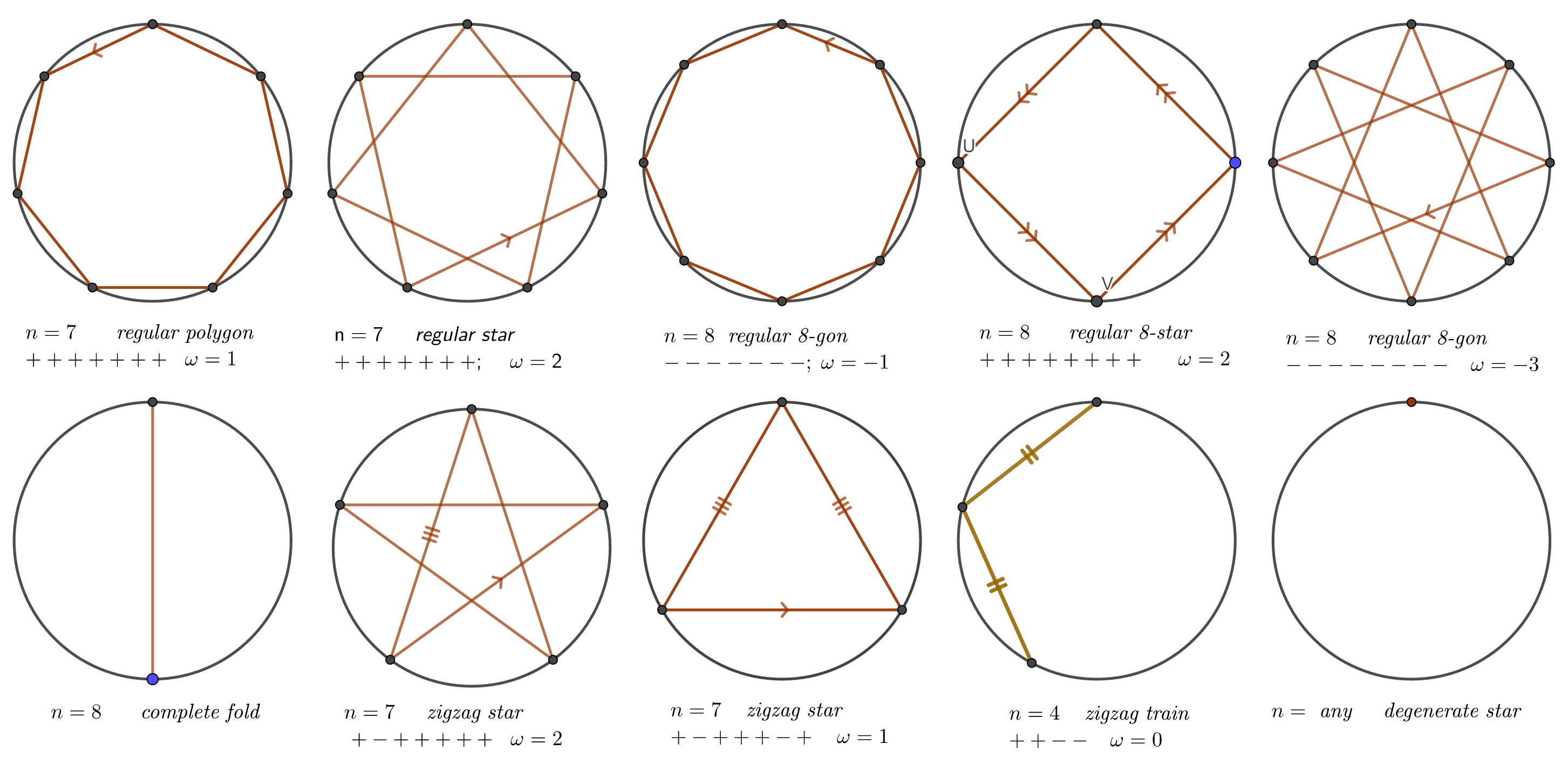}
	\caption{Some critical configurations}
	\label{fig:stars}
\end{figure}

\begin{theorem}\label{t:mainthm}

The signed area function for polygons on a circle (defined on the reduced configuration space) has critical points iff all $|\alpha_i |$ are equal. 
These critical points are 
isolated or (if the number of vertices is even) contain also a 1-dimensional singular set. Moreover
\begin{itemize}
\item[1.] The isolated singularity types are regular stars, zigzag stars and if $n=$odd also degenerate stars,
\item[2.] All regular and zigzag stars  are Morse critical points,
\begin{itemize}
\item[(i)] If $f>b$, 
the number of critical points with $ b$  backward edges is $\tbinom{n}{b}(m-b)$, each with Morse index $f-1$,
\item[(ii)]  If $f < b$ ,
the number of critical points with $ f$  forward edges is $\tbinom{n}{f}(m-f)$, each with  Morse index $f$.
\end{itemize}

\item[3.] Degenerate stars are  degenerate isolated  critical points if $n$ is odd.  Their (complex) Milnor number is $2^{n-1}$ and the real gradient map has local degree 
$\iota=  2 (-1)^m \tbinom{n-2}{m-1}$. 
\item[4.] The non-isolated case only occurs if  n = even and 
includes the complete fold, zigzag trains and degenerate stars.
The non-isolated part of the critical set contains $\tbinom{n}{\frac{n}{2}}$ branches, which meet only  at the complete fold and the degenerate stars.
\end{itemize}

\end{theorem}

\begin{proof}
The conditions for critical points are\\  $\cos \alpha_i = \cos \alpha_n$ ($ i=1,\cdots n-1$). This is equivalent to all $|\alpha_i |$ are equal.
If $f \ne b$ there are finitely many solutions; if $f=b$  infinitely many.
\begin{itemize}
\item[1.] This follows from the fact that  all $|\alpha_i |$ are equal.

\item[2.] 
First we  count the critical points with $ b$ edges. Since the polygon has $n$ edges, we have $\tbinom{n}{b}$ different possibilities.
We  also have to take into account the winding number $\omega$ of the polygon, which determines the embedding.  This winding number for $f >b$  is between $1$ and $m-b$.  This explains the factor in the statement.  The case $f<b$ is  similar.

\smallskip
Second, the Hessian matrix has diagonal term $-\sin \alpha_i - \sin \alpha_n$ and subdiagonal terms $p :=\sin \alpha_n $. 
We first consider positive regular stars; i.e. $f=n$ and $p > 0$. In this case $\alpha_1=\alpha_2= \cdots = \alpha_n = \frac{2 \pi \omega }{n}$ , where $\omega \ne 0, \frac{n}{2}$.  The matrix is:
{\small
$$             \begin{pmatrix}
               -2p & -p   &-p & \cdots  &-p \\
             -p  &  - 2p  & -p & \cdots & -p \\
             -p & -p & -2p & \cdots  & -p \\
						  \cdots & \cdots &\cdots &   & \cdots \\
						 -p & -p & -p & \cdots & - 2p\\
             \end{pmatrix}
$$
}
The eigenvalues are $-np,-p,-p, \cdots ,-p$.  It follows that the regular star (with $\omega >0) $ has a maximum if $1 \le \omega< \frac{n}{2}$ and a minimum if $ \frac{n}{2} < \omega   \le n-1 $.\\
NB. If $\omega = \frac{n}{2}$  we have a complete fold and a vanishing Hessian.

Next the zigzag case:\\
Assume $\alpha_n > 0$. If not, change the cyclic ordering such that this is the case. The Hesse matrix can change, but not the index.
The Hesse matrix has as elements $h_{ij} = -p$ if $i \ne j$ and $h_{ii}= 0$ if the $i^{th}$ edge is backwards and $h_{ii} = -2p$ if the $i^{th}$ edge is forward (this happens, resp $b$ and $f-1$ times). A typical example with $b=2$  is:
{\small
$$             \begin{pmatrix}
               0 & -p   &-p & \cdots  &-p \\
             -p  &  0  & -p & \cdots & -p \\
             -p & -p & -2p & \cdots  & -p \\
						  \cdots & \cdots &\cdots &   & \cdots \\
						 -p & -p & -p & \cdots & - 2p\\
             \end{pmatrix}
						$$
}
A direct computation shows now for any $b$ that the characteristic equation is:
$$ (\lambda^2 +(b+f-1)p \lambda + (b-f) p^2) \cdot (\lambda -p)^{b-1} \cdot (\lambda + p)^{f-2} = 0 $$ 
This explains the index formula of the theorem. Note that the index is independent of the places of the zigzags.

\smallskip
We have excluded the case where $b=f$.  This case corresponds to the nonisolated critical points and the Hessian determinant is zero (follows also from the above formula). 

\smallskip

\item[3.]  Degenerate star: ($n =2m+1)$\\
In case all $\alpha_i = 0$ it is clear that the 2-jet vanishes at the critical point. The 3-jet is  $ \frac{1}{3}( \alpha_1^3 + \cdots + \alpha_{n-1}^3 - (\alpha_1 +\cdots + \alpha_{n-1})^3)$. Therefore the Milnor number is $2^{n-1}$ ($n$ odd).

Next we will use that the sum of all indices of zero's of the gradient vector field is equal to the Euler characteristic of the configuration space; which is zero in our case.
We know the Morse indices of all non-degenerate critical points.  The index $\iota$ of the degenerate critical point
is therefore equal to:
$$\iota  = -2  \sum_{b=0}^{m-1}  (-1)^b \tbinom{n}{b}(m-b) = 2 (-1)^m \tbinom{n-2}{m-1} .$$
NB. This reduction of the first summation to a single binomial coefficient is due to repeated use of the formula $\tbinom{n}{k} = \tbinom{n-1}{k-1}+\tbinom{n-1}{k}$.

\item[4.] Non isolated singularities:\\
These occur only in case $n$ is even. Besides the degenerate star and the complete fold, where all $|\alpha_i|$ are $0$ or $\pi$ we get exactly the case $b=f$, which gives rise to zigzag trains. They can be distinguished by the positions of forward and backward edges. There are exactly $\tbinom{n}{\frac{n}{2}}$ possibilities. They correspond to branches, parametrized via $|\alpha|$   in the open interval $(0,\pi)$.The complete fold and degenerate star are limiting cases. 
\end{itemize}
\end{proof}

\newpage

\begin{example}
When $n=7$ we find the following critical polygons: (pictures in Figure \ref{fig:stars}).
\begin{itemize}
\item[b=0:]  3 maxima (index 6): regular stars with resp. $\omega = 1, 2, 3$,
\item[b=1:] 14  saddle points with index 5; resp  7 zigzag stars with $\omega=1$ and 7 zigzag stars with $\omega = 2$
(pentagonal shape),
\item[b=2:] 21 saddle points with index 4 : zigzag stars with $\omega=1$ (triangular shape),
\item[next:] a degenerate star with gradient-index  -20 ,
\item[b=5:] 21  saddle points with index 2 : zigzag stars with $\omega=-1$ (triangular shape),
\item[b=6:] 14  saddle points with index 1; resp  7 zigzag stars with $\omega= -1$ and 7 zigzag stars with $\omega =-2$
(pentagonal shape),
\item[b=7:]  3 minima (index 0): regular stars with resp. $\omega = -1, -2,-3$.
\end{itemize}
The gradient indices count up to the Euler characteristic of the configuration space:
 \begin{center} $ 3 -14 + 21  -20 + 21 -14 + 3 = 0$
\end{center}
\end{example}

\begin{proposition}
The   critical values of $sA$ are 
\begin{itemize}
\item for regular stars and zigzag star:  $(n-2b) \sin \frac{2  \pi \omega}{n-2b}$,
\item for degenerate stars, complete folds and zigzag trains: $0$.
\end{itemize}
\end{proposition}

In figure \ref{fig:extrema} we show the critical values (in case $n=7$, positive $\omega$), together with the common value $|\alpha_i|$. Negative $\omega$ results to adding a minus sign.

\begin{figure}[ht]
	\centering
	\includegraphics[width=0.6 \textwidth]{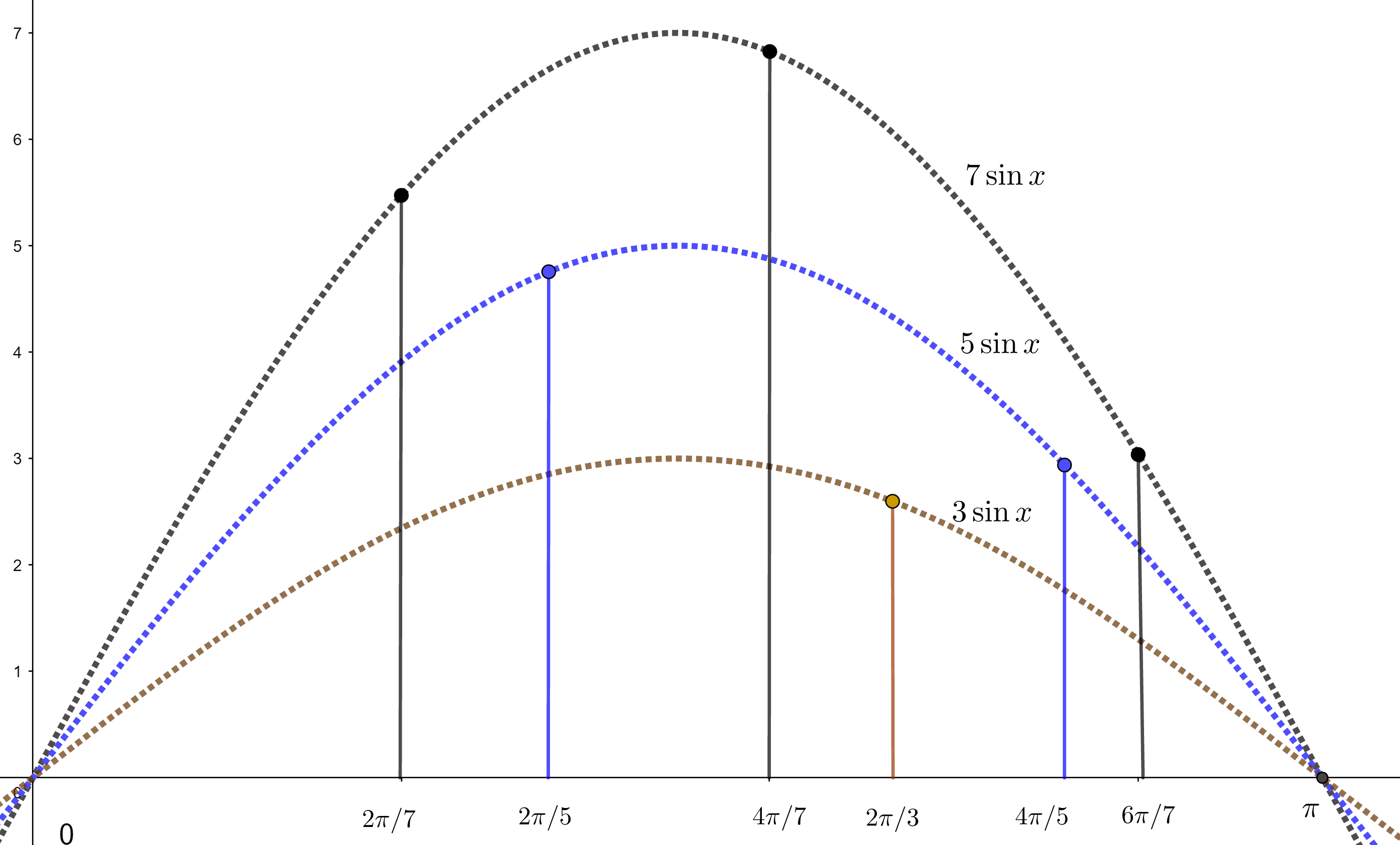}
	\caption{Extremal values, listed with$|\alpha_i |$ }
	\label{fig:extrema}
\end{figure}
Note that the absolute maximum does not occur for the regular polygon, but for a regular star with $2 \pi \omega /n$ nearest to $\pi /2$. This is because we measure Area with overlaps. All the marked points on the upper sin-curve: $y = n \sin{x}$ correspond to maxima;
on the other sin-curves $y = (n-2b) \sin{x}$  (with $b > 0$) one finds only saddle points. Drawing similar figures for $n$ even show several critical points with different $b$ but the same value. The absolute maximum is $n$ in this case.

%%%%%%%%%%%%%%%%%%%%%%%%%%%%%%%%%%%%%%%%%%%%
\section{ Eisenbud-Levine-Khimshiashvili index computations}
%%%%%%%%%%%%%%%%%%%%%%%%%%%%%%%%%%%%%%%%%%%

 In case $n=2m+1$ we know that $sA$ has an isolated critical point at a degenerate star and the Eisenbud-Levine-Khimshiashvili (ELK) signature formula is available (\cite{EL},\cite{Khim}) to compute the gradient index. This can serve as a check for correctness of our computations, but also relates to a nice combinatorial question.
For computational reasons we replace Area $sA$ by its 3-jet: $f = j_3(sA)$ and compute its index. This does not  change the index, see \cite{CGT}.

\subsection{Milnor Algebra}
\subsubsection{The general case}
We first give a sketch of the ELK method for the gradient vector field. Let $f: \mathbb{R}^n ,0 \to \mathbb{R}$ be a real analytic function, such that the complexification has an isolated zero.
Let $J(f)=(f_1,\cdots,f_n)$ be the ideal generated by the partial derivatives of $f$ in the ring $O_{n,O}$ of real analytic function germs. Let $M_f =O_{n,0} /  J(f)$ , this is the (real) Milnor Algebra of $f$. Its dimension is finite if and only if the complexification of $f$ has an isolated singularity. Let $h_f$ be the Hessian determinant of  the germ $f$ and $[h_f] \in M_f$ its equivalence class modulo $J(f)$.

\smallskip
\noindent 
Let $l$ be  any linear function such that  $l ([h_f]) > 0$.
Define a nondegenerate bilinear form $\beta$ on $M_f$ by $\beta(g,h) = l (g * h)$, where $*$ is the multiplication in $O_{n,0}$. The ELK signature formula states that the index of the gradient vector field  at the origin is equal to the signature of the bilinear form $\beta$ on $M_f$, i.e the number of positive eigenvalues minus the number of negative eigenvalues.\\
In order to compute the signature we work as follows
\begin{itemize}
\item[-] choose a set of function germs, which induce  a basis of $M_f$,
\item[-] write down the multiplication table  (modulo $J(f))$ on this basis,
\item[-] compute $[h_f]$ and choose the linear function $l$ as above,
\item[-] substitute the $l$-values in the multiplication table; this gives a scalar valued symmetric matrix, compute its signature.
\end{itemize}

\subsubsection{The case of our 3-jet}

The Milnor algebra $M_f$ is graded by degree and has dimension $2^{n-1}$. The partial derivatives are homogeneous of degree $2$ and give rise to the relations:
$$ x_1^2 = x_2^2=\cdots = x_{n-1}^2 = \tfrac{-2}{n-2} \sum_{1 \le  i < j \le n-1} x_i x_j  $$
(we changed notation and use $x_i$ in stead of $\alpha_i$).
Recall $n= 2m + 1 $, the singularity is isolated and the dimension of the degree $d$ part of $M_f$ is $\binom{n-1}{d}$. A basis (modulo the partial derivatives) are given by the monomials without repeated factors.

For our purpose we need to understand especially the part of maximal degree $n-1$, which is 1-dimensional and to determine the relations between the monomials of degree $n-1$. We first look to the effect in this degree of the relations
 $x_1^2 = x_2^2=\cdots = x_{n-1}^2$
mentioned  above. It follows that the class $[x^I] = [ x_1^{i_1} x_2^{i_2} \cdots x_{n-1}^{i_{n-1}}]$ is completely determined by the unordered $(n-1)$-tuple $(i_1,i_2,\cdots,i_{n-1})$ modulo $2$.  E.g. $I_k$ consisting of $2k$ tuples equal $1$ and the others equal $0$. Let the corresponding class of monomials be  $w_k = [ x^{I_k}],  (k=0, \cdots m)$. Note $w_m =[ x_1 x_2 \cdots x_{2m}].$

Next we add the remaining relation $x_p^2 =  \tfrac{-2}{n-2} \sum_{1 \le  i < j \le n-1} x_i x_j $. \\
The effect in degree $n-1$ is the following set of relations ($p = 0, \cdots, m-1$) :
$$ \binom{2p}{2}  w_{p-1} + \{\tfrac{n-2}{2}  + 2p(n-1-2p)\} w_{p} +  \binom{n-1-2p}{2} w_{p+1}= 0 . $$
The solution space is 1-dimensional and determined by the following recurrence  relations:
$$ w_{p-1} = -  \frac{2m+2-2p}{2p-1} w_p  \;  \;  \; \; \; \; (p=1,\cdots,m),$$ 
which determines
$$ w_k = (-1)^{m-k} \frac{(2m-2k)!! (2k-1)!!}{ (2m-1)!!}w_m  \;  \;  \; \; \; \; (k=0,\cdots,m).$$

\noindent
Next:  $h_f = x_1 x_2 \cdots x_{n-1} - \sum_{1 \le i ,j \le n-1}  x_i  \frac{x_1 x_2 \cdots x_{n-1}}{x_j}$. In $M_f$  we have $$[h_f] = - (n-2) w_m - (n-1)(n-2)w_{m-1} = -(n-2) w_m + 2 (n-1)w_m = n \, w_m.$$

We fix $l(w_m) = \omega $. This determines al the other $l(w_k)$.  All other basis elements of $M_f$ have lower degree.
We set $l=0$ for all of them.

\noindent
{\bf Examples}\\
$ m=1   \; : \:   w_0  = -2 \omega \; , \; w_1 = \omega $ \\
$ m=2   \; : \:   w_0  = \frac{8}{3} \omega \; , \; w_1 = - \frac{2}{3} \omega \;  , \;  w_ 2 = \omega $ \\
$ m=3   \; : \:   w_0  = - \frac{16}{5} \omega \; , \; w_1 =  \frac{8}{15} \omega \;  , \;  w_ 2 = - \frac{2}{5} \omega \; , \; w_3 = \omega$ 

\subsubsection{
The multiplication matrix}

\smallskip
Next consider the matrix $B$ of $l(\beta)$  in block form  $(B_{i,j})$  where $B_{i,j}$ is the submatrix constructed by multiplying monomials of degree $i$ with those of degree $j$. Due to symmetry we have $B_{i,j}=B_{j,i}$. Note that $B_{i,j} = 0$ if $i+j \ge n$ . Due to the choice of $l$ we also have that $B_{i,j} = 0$  if $i+j < n-1$.

\smallskip
The matrix $B$ has now the  following anti-diagonal block form: (e..g. $n=5$)
$$
\begin{pmatrix}
0 & 0 & 0 & 0 & B_{0,4} \\
0 & 0 & 0 & B_{1,3} & 0  \\
0  & 0 & B_{2,2} & 0 & 0 \\
0 & B_{3,1} & 0 & 0 & 0 \\
B_{4,0} & 0 & 0 & 0 & 0 
\end{pmatrix}
$$
Next observe that due to the symmetry: signature $B$ = signature $B_{m,m}$.

\subsection{A combinatorial  question}
The computation of the  signature of $B_{m,m}$ is related to an eigenvalue question of combinatorial type. This question was solved in a manuscript of the paper for m=1,2,3 only. Due to the suggestions of an anonymous referee and the assistance of Wilberd van der Kallen it was possible to give the full solution. We will publish the combinatorial result (in a more general case) in a separate paper  \cite{vKS} and summarize below the results. Relevant notions are the representation theory of the symmetric group and the relation between invariant subspaces and the eigenspaces of the intertwining matrices. As references we mention \cite{Fu} and \cite{Pr}.

\medskip
\noindent
Consider combinations $\mathcal{C}^{2m}_{m}$ of $m$ elements out of a set of $2m$ elements. Take an arbitrary tuple of real numbers $b_0,\cdots ,b_m$. We constitute a matrix, where the rows and columns are indexed (in lexicografic order) by elements of $\mathcal{C}^{2m}_m$.
The matrix elements are defined as follows: 
\begin{center}
$ < \sigma, \tau> = b_p$ if  $ \sigma \cap \tau $ has $p$ elements ($\; 0 \le p \le m$).
\end{center}

\begin{figure}[ht]
	\centering
\begin{minipage}{0.49\textwidth}
\centering
	\includegraphics[width=0.4 \linewidth]{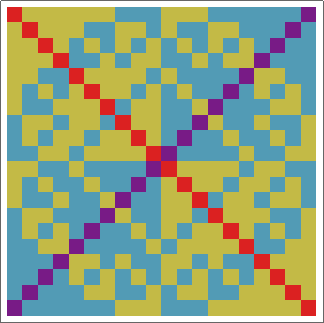}
\end{minipage}
\begin{minipage}{.49\textwidth}
\centering
	\includegraphics[width=0.4 \linewidth]{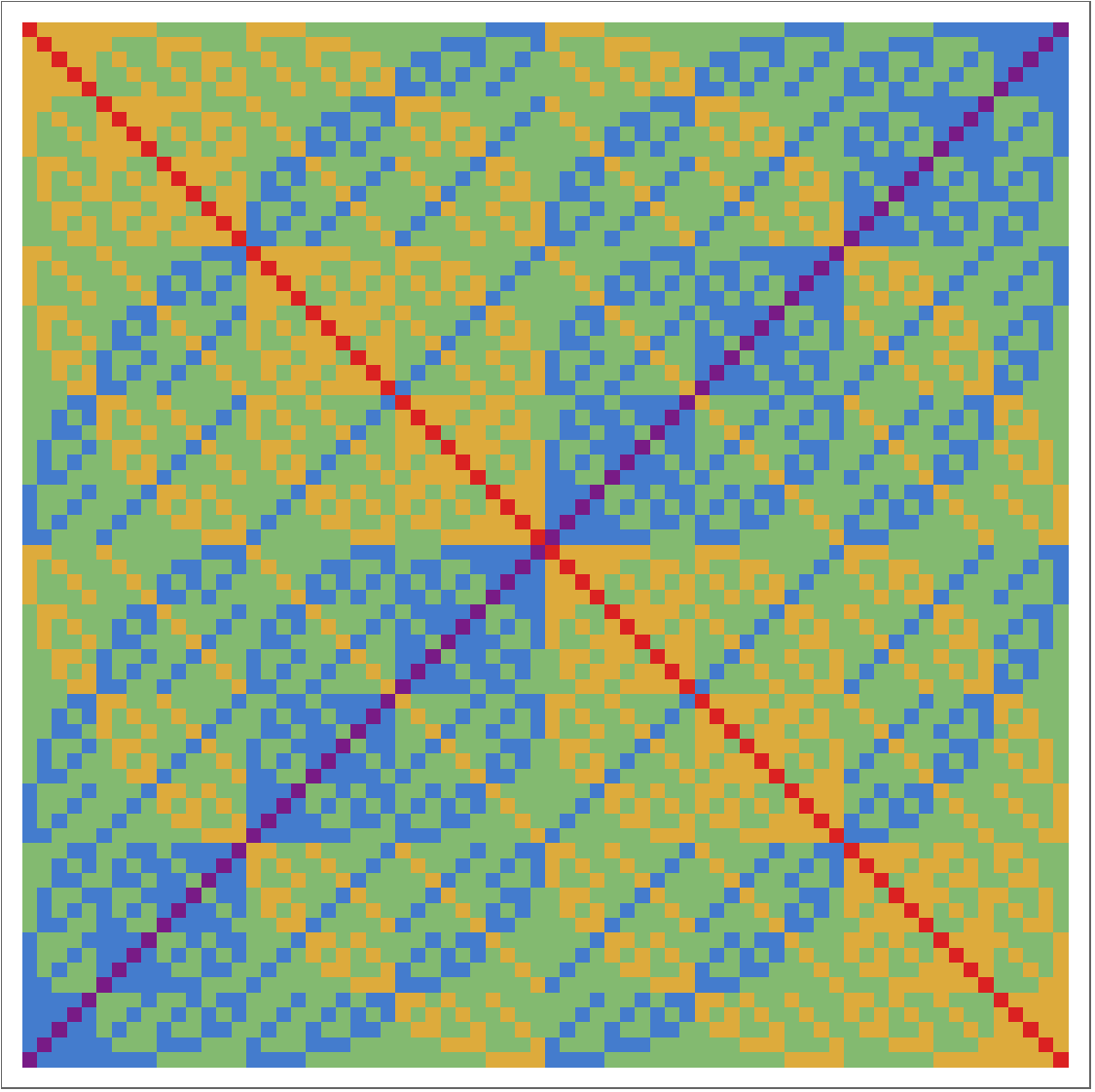}
\end{minipage}
	\label{fig:array}
\caption{$B_{2,2}$ and $B_{3,3}$;  coloured by $b_p$.}
\end{figure}

\newpage

\begin{proposition} The matrix $B_{m,m}$ has the properties:
\begin{itemize}
\item The eigenspaces  are independent of a (generic) choice of $b_0,\cdots b_m$,
\item The eigenvalues are $\mathbb Z$-linear combinations of  $b_0,\cdots b_m$.
\end{itemize}
\end{proposition}
This results is clear from Schur's lemma and Young's rule as the referee suggested. His remark about Specht modules made it possible to compute the eigenvalues. That work was carried out  in \cite{vKS}.

\begin{proposition}  The matrix $B_{m,m}$ has eigenvalues
$$ 
\lambda_k = \sum _{j=0}^k \sum _{p=0}^{m-k} (-1)^{k-j} \binom{k}{j}  \binom{m-j}{p} \binom{j-k-m+n}{-k+m-p} b_{j+p}
$$
each of multiplicity 1 with dimension of eigenspaces:
$\mu_k = \frac{(2m)! (2m-2k+1)}{k! (2m-k+1)!}$.
\end{proposition}
\smallskip
\begin{example}

$B_{1,1}$ has eigenvalues:
\begin{itemize}
\item $b_0 +b_1$ with multiplicity 1, and
\item $b_1 - b_0$  with multiplicity 1.
\end{itemize}

$B_{2,2}$ has eigenvalues:
\begin{itemize}
\item $b_0  + 4b_1 + b_2$ with multiplicity 1,
\item $b_0  - 2b_1   + b_2$ with multiplicity 2 and
\item $b_2 - b_0$  with multiplicity 3.
\end{itemize}

$B_{3,3}$ has eigenvalues:
\begin{itemize}
\item $ b_0 + 9 b_1 + 9 b_2 +b_3$ with multiplicity 1,
\item $ b_0 - b_1 - b_2 + b_3$ with multiplicity 9,
\item  $ -b_0 + 3 b_1 - 3 b_2 + b_3$  with multiplicity 5 and
\item $-b_0 -3 b_1 + 3b_2 + b_3$  with multiplicity 5. 
\end{itemize}
\end{example}

\medskip
\subsection{Final step; computation of the index}

 \begin{proposition}\label{three}
 The restriction of $B_{m,m}$ to the 1-dimensional degree m part of the Milnor algebra  has the eigenvalues 
$$\lambda_k  = (-1)^m \frac{2m+1}{2m-2k+1}  \omega  \;  \;  (k=0,\cdots,m)$$
\end{proposition}

\begin{proof}
Take  $l(w_p) = b_{m-p}$.
 Note 
 $ b_p = (-1)^{p} \omega\frac{(2p)!! (2m-2p-1)!!}{ (2m-1)!!} \;  \;  (p=0,\cdots,m) .$ 
 We insert this in the formula for $\lambda_k$ in Proposition \ref{three}. 
 The formula in the present Proposition, which is in fact an enormous reduction, was expected by numerical results.
We  discovered  a key  recurrence  using the package \cite{weg}. See Appendix. The proof of the formula  becomes straight forward after checking the initial cases.
\end{proof}
 
\smallskip
  
  \begin{example} Set $\omega = 1$; \\
  $m=1$ eigenvalues  $-1,-3$,\\
 $ m=2$ eigenvalues $1, \frac{5}{3}, 5$,\\
  $m=3$ eigenvalues $-1,-\frac{7}{5},-\frac{7}{3},-7$,\\
  $m=8$ eigenvalues $1,\frac{17}{15},\frac{17}{13},\frac{17}{11},\frac{17}{9},\frac{17}{7},\frac{17}{5},\frac{17}{3},17.$
  
  \end{example}
   
Set $\omega = 1$ and conclude that the eigenvalues of $B_{m,m}$ have the same sign and the index is equal to  $\iota=  2 (-1)^m \tbinom{2m-1}{m-1}$. 
 This coincides with the result  in Theorem \ref{t:mainthm}.

\newpage

\section*{Appendix} 
We discovered the key identity (\ref{key}) below using \\
 Package MultiSum version 2.3 written by Kurt Wegschaider\\
enhanced by Axel Riese and Burkhard Zimmermann\\
Copyright Research Institute for Symbolic Computation (RISC),\\
Johannes Kepler University, Linz, Austria.\\
   
\small{
Recall that $(-1)!!=0!!=1$ and $n!!=n((n-2)!!)$ for $n\geq1$.

If $0\leq p\leq m-k$, $0\leq j\leq k$, put
$$F(m,k,j,p)=\frac{(-1)^{k+p} \binom{k}{j}\binom{m-j}{p} \binom{j-k+m}{j+p}(2 (j+p))\text{!!}  (-2 j+2 m-2
   p-1)\text{!!} }{(2 m-1)\text{!!}}.$$
Otherwise put $F(m,k,j,p)=0$.

Put 
\begin{align*}
Fj(m,k,j,p)=\ &(2 m-7)(k-1) (
\\&-4 (k-2) (k+m-4) F(m-3,k-3,j,p+1)
\\&+4 (m-2) (k-m)
   F(m-3,k-2,j,p)
\\&   -4 (m-2) (k-m) F(m-3,k-2,j,p+1)
\\&   -(2 m-5) (4 k+2 m-9)
   F(m-2,k-2,j,p+1)
\\&   -(2 k-3) (2 m-5) F(m-2,k-2,j+1,p+1)
\\&   +2 (2 m-5) (k-m)
   F(m-2,k-1,j,p)
\\&   +2 (2 m-5) (m-k) F(m-2,k-1,j,p+1)
\\&   +2 (2 m-5) (k-m)
   F(m-2,k-1,j+1,p)
\\&   +2 (2 m-5) (m-k) F(m-2,k-1,j+1,p+1)
\\&   -(2 m-5) (2 m-3)
   F(m-1,k-1,j,p+1)
\\&   -(2 m-5) (2 m-3) F(m-1,k-1,j+1,p+1))
\end{align*}
    
and
\begin{align*} 
   Fp(m,k,j,p)=&(k-1)(
\\&   -4 (k-3) (k-2) (2 m-5) F(m-4,k-4,j,p)
\\&   -4 (k-2) (2 m-7) (k+2 m-6)
   F(m-3,k-3,j,p)
\\&   +4 (m-2) (2 m-7) (m-k) F(m-3,k-2,j,p)
\\&   -(2 m-7) (2 m-5)
   (4 k+2 m-9) F(m-2,k-2,j,p)
 \\&  +2 (2 m-7) (2 m-5) (m-k) F(m-2,k-1,j,p)
\\&   -(2
   m-7) (2 m-5) (2 m-3) F(m-1,k-1,j,p)).
\end{align*}

%\theequation

Then we have the key identity
\begin{align*}& 
4 (k-3) (k-2) (k-1)(2 m-5) F(m-4,k-4,j,p)
\\&\refstepcounter{equation}\tag{\theequation}\label{key}
+4 (k-2)(k-1) (2 m-7) (k+2 m-6)
   F(m-3,k-3,j,p)
\\&   +(2 m-7) (2 m-5)(k-1) (4 k+2 m-9) F(m-2,k-2,j,p)
\\&   +(2 m-7) (2
   m-5) (2 m-3)(k-1) F(m-1,k-1,j,p)
\\&=Fj(m,k,j+1,p)-Fj(m,k,j,p)+Fp(m,k,j,p+1)-Fp(m,k,j,p). 
   \end{align*}
 %\theequation
  
It yields a recurrence for the $\sum_{j,p}F(m,k,j,p)$.

Similarly, put 
\begin{align*} g(m,p)=\ &4 m^2 F(m-1,0,0,p-2)
\\&-(2 m-1) (2 m+1) F(m,0,0,p-2)
\\&+(2 m-1) (2 m+1)
   F(m,0,0,p-1)
\\&   -(2 m-1) (2 m+1) F(m+1,0,0,p-2)
\\&   -(2 m-1) (2 m+1)
   F(m+1,0,0,p-1).
    \end{align*}

   Then $(2 m-1) (2 m+1) F(m,p-2)+(2 m-1) (2 m+1) F(m+1,p-2)
    =g(m,p+1)-g(m,p)$, which yields a recurrence for the $\sum_{p}F(m,0,0,p)$.

One concludes that  $\sum_{p}F(m,0,0,p)=(-1)^m$ for $m\geq0$ and then that
$$\sum_{j,p}F(m,k,j,p)=(-1)^m (2 m + 1)/(2 (m - k) + 1)$$ for $m\geq k\geq0$.
}
 \end{document}